\documentclass[a4paper,12pt]{article}
\usepackage[left=2cm, right=2cm, top=2.50cm, bottom=2.50cm]{geometry}
\usepackage{cite, amsmath, amssymb}
\usepackage{graphicx}
\usepackage{amsmath,amsthm,amssymb,lineno}
\usepackage{latexsym}
\usepackage{graphicx,booktabs,multirow}
\usepackage{tikz}
\usetikzlibrary{decorations.pathreplacing}
\usetikzlibrary{intersections}
\usepackage{enumerate}
\usepackage{graphicx,booktabs,multirow}
\usepackage{appendix}

\allowdisplaybreaks[4]
\newtheorem{theo}{Theorem}[section]

\newtheorem{cor}[theo]{Corollary}
\newtheorem{prop}[theo]{Proposition}

\begin{document}

\title{The expected values of Sombor indices in random hexagonal chains, phenylene chains and Sombor indices of some chemical graphs \thanks{This work is supported by the National Natural Science Foundation of China (Grant No. 11971180), the Guangdong Provincial Natural Science Foundation (Grant No. 2019A1515012052).}}
\author{Xiaona Fang, Lihua You\thanks{Corresponding author: ylhua@scnu.edu.cn}, Hechao Liu\\
{\small School of Mathematical Sciences, South China Normal University,}\\ {\small Guangzhou, 510631, P. R. China}}
\date{}
\maketitle

{\bf Abstract}: Hexagonal chains are a special class of catacondensed benzenoid system and phenylene chains are a class of polycyclic aromatic compounds. Recently, A family of Sombor indices was introduced by Gutman in the chemical graph theory. It had been examined that these indices may be successfully applied on modeling thermodynamic properties of compounds. In this paper, we study the expected values of the Sombor indices in random hexagonal chains, phenylene chains, and consider the Sombor indices of some chemical graphs such as graphene, coronoid systems and carbon nanocones. 


\baselineskip=0.30in

\section{Introduction}
\hspace{1.5em}
Chemical graph theory is an interdisciplinary field of science which relates chemistry with a branch of mathematical modeling of graphs. Topological indices are graph invariants that play an important role in chemical and pharmaceutical sciences, since they can be used to predict physicochemical properties of organic compounds \cite{tri2018}. There are lots of topological indices in the literature of chemical graph theory. Recently, Gutman introduces a family of Sombor indices in the chemical graph theory \cite{gumn2021}. It was examined in \cite{redz2021} that the Sombor index, reduced Sombor index and average Sombor index showed satisfactory predictive and discriminative potential in modeling entropy and enthalpy of vaporization of alkanes. The results of testing predictive potential of Sombor indices indicate that these descriptors may be successfully applied on modeling thermodynamic properties of compounds.

Let $G=(V,E)$ be a finite, connected, simple graph with vertex set $V=V(G)$ and edge set $E=E(G)$, where $|V(G)|$ is the number of vertices and $|E(G)|$ is the number of edges. We denote the degree of a vertex $i$ in $G$ by $d_i$. The (ordinary) Sombor index is defined as
    $$SO(G)= \sum\limits_{i \sim j}\sqrt{d_i^2+d_j^2},$$
the reduced Sombor index is defined as
 $$SO_{red}(G)= \sum\limits_{i \sim j}\sqrt{(d_i-1)^2+(d_j-1)^2},$$  
 and the average Sombor index, as
$$SO_{avr}(G)= \sum\limits_{i \sim j}\sqrt{(d_i-\bar{d})^2+(d_j-\bar{d})^2},$$  
where $\bar{d}=\frac{2\cdot |E(G)|}{|V(G)|}$ is the average degree of graph $G$ \cite{gumn2021}. In this paper, Sombor indices refer to Sombor index, reduced Sombor index and average Sombor index.
Let $a$ be any real number or parameter of graph $G$. We generalize the Sombor indices with $a$. The generalized index is defined as
\begin{equation}\label{soma}
	SO_{a}(G)= \sum\limits_{i \sim j}\sqrt{(d_i-a)^2+(d_j-a)^2}.
\end{equation} 
It's clear when $a=0$, $SO_{a}(G)=SO(G)$, when $a=1$, $SO_{a}(G)=SO_{red}(G)$ and when $a=\bar{d}$, $SO_{a}(G)=SO_{avr}(G)$.

Sombor indices have attracted much attention due to good chemical applicability. Cruz, Gutman and Rada characterized the extremal graphs of the chemical graphs, chemical trees and hexagon systems with respect to Sombor index \cite{crug2021}. In \cite{alig2021}, the Sombor index of polymer graphs which can be decomposed into monomer units was considered. In \cite{crur2021}, the extremal values of the Sombor index in unicyclic and bicyclic graphs were studied. Das, Cevik, Cangul and Shang presented lower and upper bounds on the Sombor index of graphs by using some graph parameters and obtain several relations on Sombor index with the first and second Zagreb indices of graphs \cite{dasc2021}. More results of Sombor indices can be found in \cite{dasg2021,dengt2021,gumn2021,guma2021,horx2021,kulg2021,liuc2021,liuy2021,liuyh2021,milo2021,redz2021,trdo2021,wmlf2021}. In Section \ref{sec2}, we study the expected values of the Sombor indices in the random hexagonal chains and random phenylene chains, and make a comparison between the expected values. In Section \ref{sec3}, we study the Sombor indices of some graphs that are of importance in chemistry such as graphene, coronoid systems and carbon nanocones, and give numerical comparison of the Sombor indices and graphical profiles of the comparison.

\section{The expected values of Sombor indices in random hexagonal chains and phenylene chains}\label{sec2}

\hspace{1.5em}
Random molecular graphs are of great importance for theoretical chemistry. There are many results about the extremal values of topological indices of random molecular graphs in recent years \cite{jaha2020,liuh2021,raza2020,raza2020b,yangz2012}. In this section, we study the expected values of the Sombor indices in random hexagonal chains and phenylene chains.

We say an edge is $(i,j)$-type if it joins a vertex with degree $i$ and a vertex with degree $j$ in $G$. Let $m_{ij}(G)$ be the number of edges of $(i,j)$-type. Then we have the following Proposition.

\begin{prop}\label{main}
	Let $G$ be a graph. If there exists only $(2,2)$, $(2,3)$ and $(3,3)$-type of edges in $G$, then we have
	\begin{equation}\label{sored}
	SO_{a}(G)=\sqrt{2}\cdot |2-a| \cdot m_{22}(G)+\sqrt{2a^2-10a+13} \cdot m_{23}(G)+\sqrt{2}\cdot |3-a|\cdot  m_{33}(G).  
	\end{equation} 	
\end{prop}
\begin{proof}
	Since there exists only $(2,2)$, $(2,3)$ and $(3,3)$-type of edges in $G$, by (\ref{soma}), we have 
	\begin{align*}
		SO_{a}(G)&= \sqrt{(2-a)^2+(2-a)^2} \cdot m_{22}(G)+\sqrt{(2-a)^2+(3-a)^2} \cdot m_{23}(G)\\
		&+\sqrt{(3-a)^2+(3-a)^2} \cdot  m_{33}(G) \\
		&=\sqrt{2}\cdot |2-a| \cdot m_{22}(G)+\sqrt{2a^2-10a+13} \cdot m_{23}(G)+\sqrt{2}\cdot |3-a|\cdot  m_{33}(G). 
	\end{align*}
	
	The proof is completed.
\end{proof}

\subsection{Random hexagonal chains}
\hspace{1.5em}
A benzenoid system is a finite connected subgraph of the infinite hexagonal lattice without cut vertices or non-hexagonal interior faces. A benzenoid system without any hexagon which has more than two neighboring hexagons is called a hexagonal chain, denoted by $HXG_n$. For $n\geq 3$, the terminal hexagon can be attached in three ways, which results in the local arrangements, we describe as $HXG_n^{1}$, $HXG_n^{2}$, and $HXG_n^{3}$, respectively, see Figure \ref{g2}.

A random hexagonal chain $HXG(n; p_{1},p_{2})$ with $n$ hexagons is a hexagonal chain obtained by stepwise addition of terminal hexagons. At each step $t(=3,4,\cdots,n)$, a random selection is made from one of the three possible constructions: \\
$(1)$ $HXG_{t-1}\rightarrow HXG_{t}^{1}$ with probability $p_{1}$; \\
$(2)$ $HXG_{t-1}\rightarrow  HXG_{t}^{2}$ with probability $p_{2}$; \\
$(3)$ $HXG_{t-1}\rightarrow  HXG_{t}^{3}$ with probability $1-p_{1}-p_{2}$, where $p_{1}$, $p_{2}$ are constants, irrelative to the step parameter $t$.
\begin{figure}[h]
    \centering
	\includegraphics[scale=0.1]{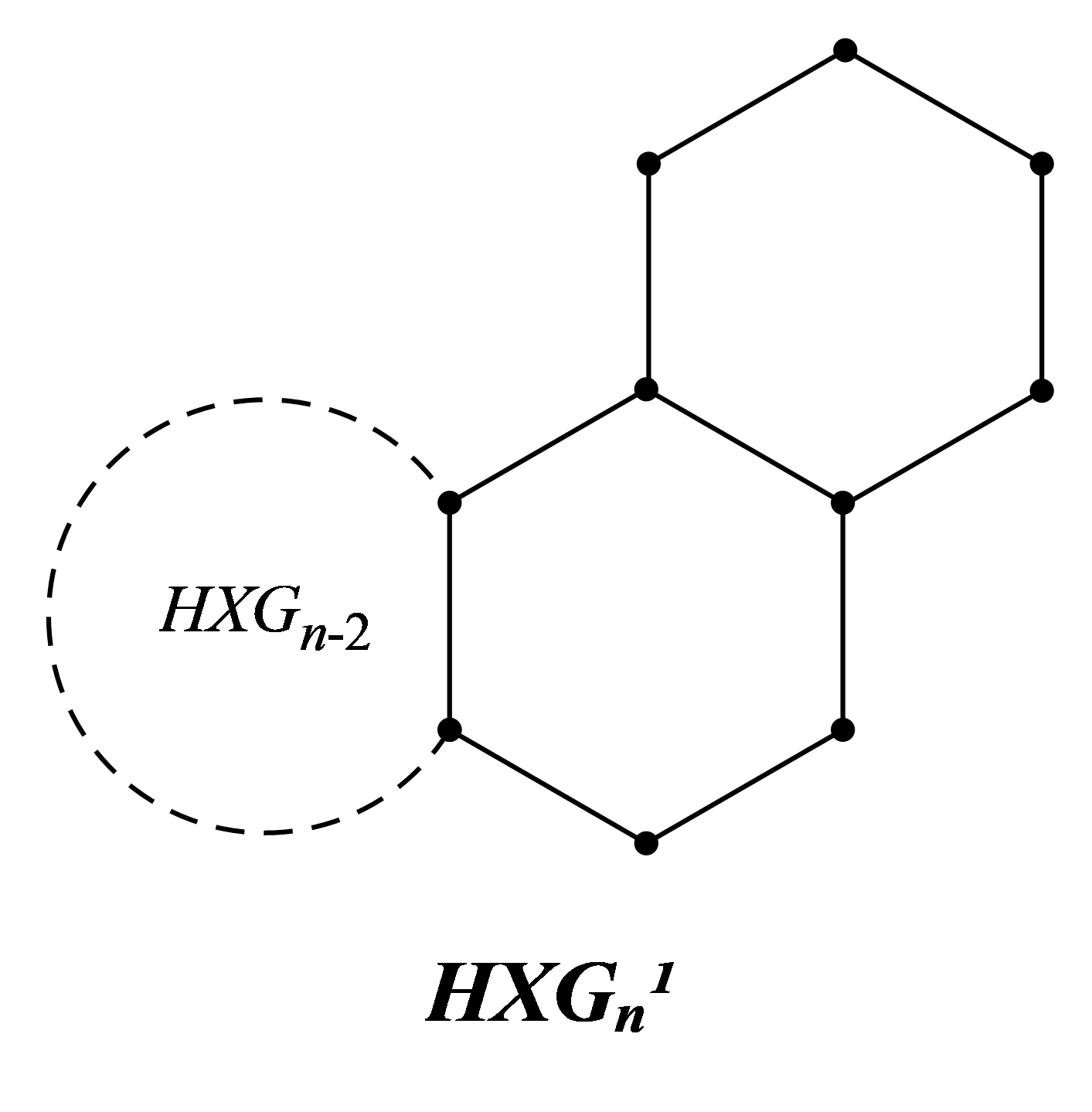}
	\qquad
	\includegraphics[scale=0.1]{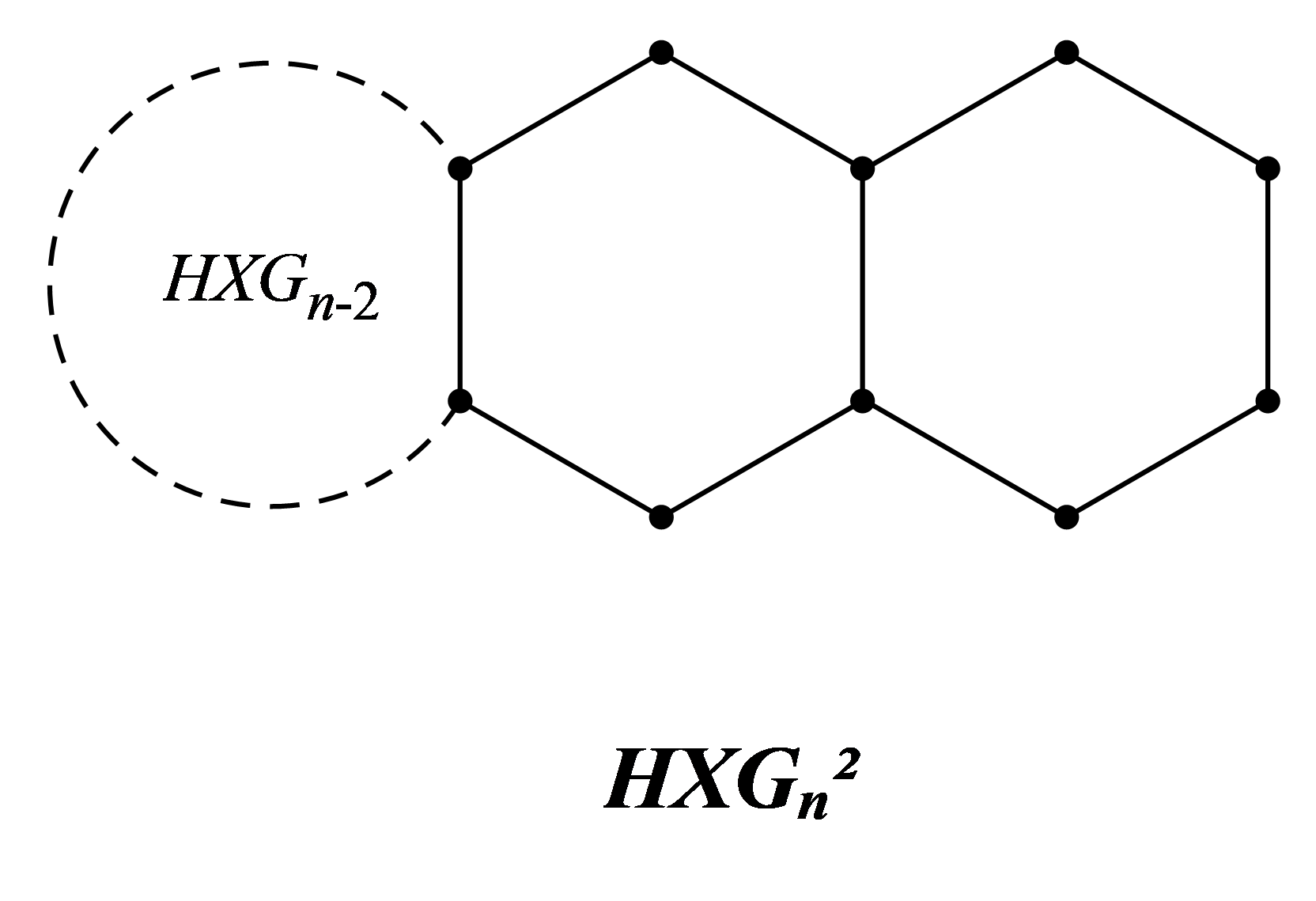}
	\qquad
	\includegraphics[scale=0.1]{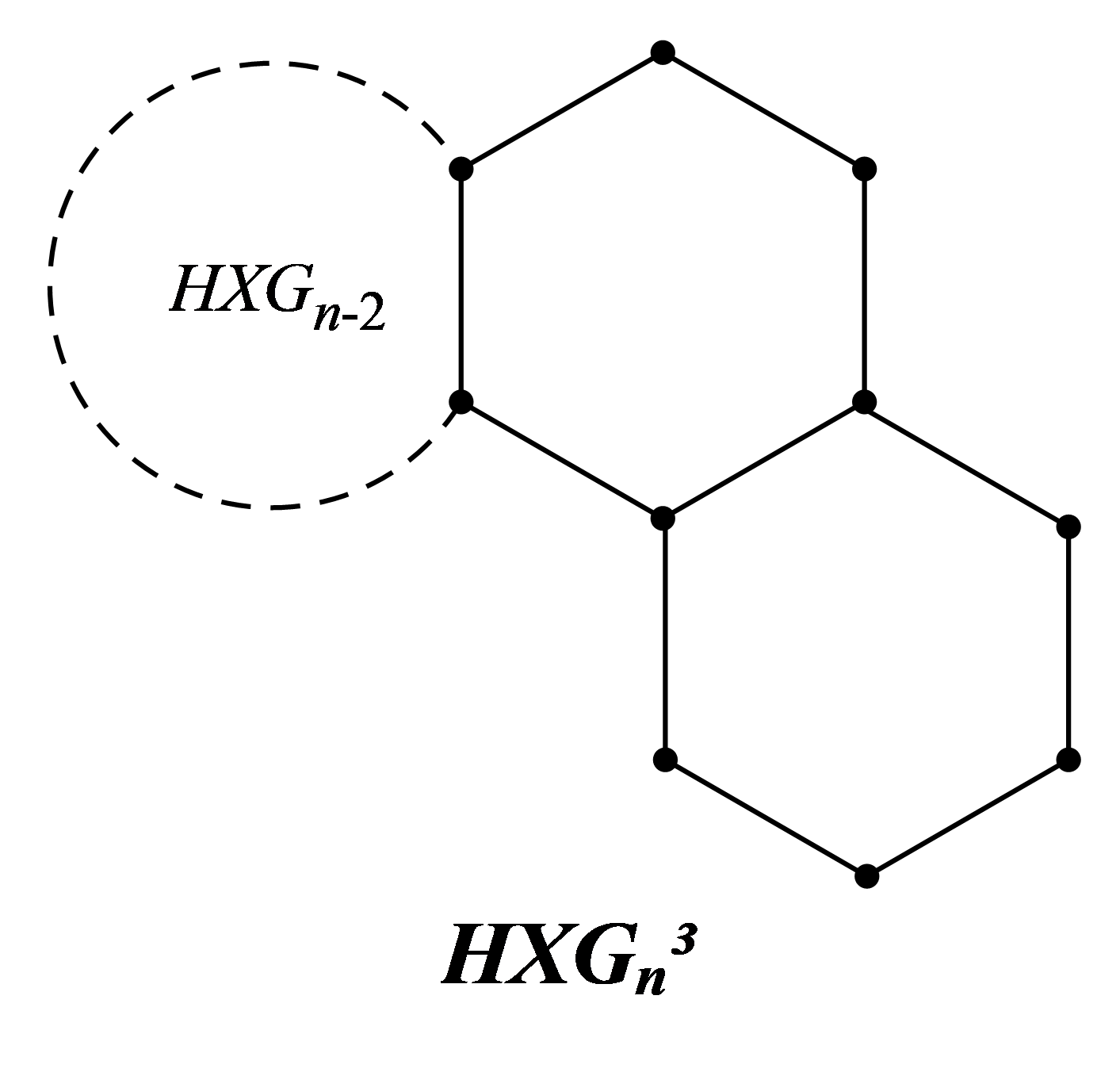}
	\caption{The three types of local arrangements in the hexagonal chains.}
    \label{g2}
\end{figure}

Since $HXG(n; p_1,p_2)$ is a random hexagonal chain, $SO(HXG(n; p_1,p_2))$, $SO_{red}(HXG(n; p_1,p_2))$ and $SO_{avr}(HXG(n; p_1,p_2))$ are random variables. We denote the expected values of these indices by $E_n^a=E[SO_a(HXG(n; p_1,p_2))]$. When $a=0$, $E_n^a=E_n$, when $a=1$, $E_n^a=E_{n}^{red}$ and when $a=\bar{d}$, $E_n^a=E_{n}^{avr}$.
In this section, $a$ is a constant.

\begin{theo}\label{hxg}
	Let $HXG(n; p_1,p_2)$ be the hexagonal chain of length $n(\geq 2)$. Then
\begin{equation}\label{hxga}
\begin{split}
E_n^a=2(np_2-2p_2+n)&\sqrt{2a^2-10a+13}+\sqrt{2}(-np_2+2p_2+n+4)\cdot |2-a|\\
&+\sqrt{2}(-np_2+2p_2+2n-3)\cdot |3-a|,
\end{split}	
\end{equation}
	\begin{equation}\label{hxgso}
	E_n=[(2\sqrt{13}-5\sqrt{2})p_2+8\sqrt{2}+2\sqrt{13}]n+(10\sqrt{2}-4\sqrt{13})p_2-\sqrt{2},
	\end{equation}
	\begin{equation}\label{hxgred}
		E_n^{red}=[(2\sqrt{5}-3\sqrt{2})p_2+5\sqrt{2}+2\sqrt{5}]n+(6\sqrt{2}-4\sqrt{5})p_2-2\sqrt{2},
	\end{equation}
\begin{equation}\label{hxgavr}
	\begin{split}
	&E_n^{avr}=[2(p_2+1)\sqrt{2\bar{d}^2-10\bar{d}+13}+\sqrt{2}(4-p_2-\bar{d})]n \\
	-4p_2 & \sqrt{2\bar{d}^2-10\bar{d}+13}+\sqrt{2}(2p_2+7\bar{d}-17),\  where \ \bar{d}=\frac{5n+1}{2n+1}.
	\end{split}
\end{equation}
\end{theo}

\begin{proof}
From the structure of the hexagonal chain, it is easy to see that there exists only $(2,2)$, $(2,3)$ and $(3,3)$-type of edges. From Proposition \ref{main}, when $n=2$, $E_2^a=4\sqrt{2a^2-10a+13}+\sqrt{2}(6\cdot |2-a|+|3-a|)$. 

For $n\geq 3$, there are three possibilities to be considered (see Figure \ref{g2}). 
	
	\textbf{Case 1}. $HXG_{n-1}\rightarrow HXG_n^1$.
	$$m_{22}(HXG_n^1)=m_{22}(HXG_{n-1})+1;$$
		$$m_{23}(HXG_n^1)=m_{23}(HXG_{n-1})+2;$$
			$$m_{33}(HXG_n^1)=m_{33}(HXG_{n-1})+2.$$
	Thus, $SO_{a}(HXG_n^1)=SO_{a}(HXG_{n-1})+2\sqrt{2a^2-10a+13}+\sqrt{2}( |2-a|+2\cdot |3-a|).$
	
	\textbf{Case 2}. $HXG_{n-1}\rightarrow HXG_n^2$.
	$$m_{22}(HXG_n^2)=m_{22}(HXG_{n-1})+0;$$
	$$m_{23}(HXG_n^2)=m_{23}(HXG_{n-1})+4;$$
	$$m_{33}(HXG_n^2)=m_{33}(HXG_{n-1})+1.$$
	Thus, $SO_{a}(HXG_n^2)=SO_{a}(HXG_{n-1})+4\sqrt{2a^2-10a+13}+\sqrt{2}\cdot |3-a|.$
	
	\textbf{Case 3}. $HXG_{n-1}\rightarrow HXG_n^3$.
	$$m_{22}(HXG_n^3)=m_{22}(HXG_{n-1})+1;$$
	$$m_{23}(HXG_n^3)=m_{23}(HXG_{n-1})+2;$$
	$$m_{33}(HXG_n^3)=m_{33}(HXG_{n-1})+2.$$
	Thus, $SO_{a}(HXG_n^3)=SO_{a}(HXG_{n-1})+2\sqrt{2a^2-10a+13}+\sqrt{2}( |2-a|+2\cdot |3-a|).$
	
	Therefore, $E_{n}^a=p_1 \cdot SO_a(HXG_n^1)+p_2\cdot  SO_a(HXG_n^2)+(1-p_1-p_2)\cdot SO_a(HXG_n^3)=SO_a(HXG_{n-1})+2(p_2+1)\sqrt{2a^2-10a+13}+\sqrt{2}[(1-p_2)\cdot |2-a|+(2-p_2)\cdot |3-a|]$. Since $E[E_n^a]=E_n^a$, we have
	\begin{equation}\label{re2}
	E_n^a=E_{n-1}^a+2(p_2+1)\sqrt{2a^2-10a+13}+\sqrt{2}[(1-p_2)\cdot |2-a|+(2-p_2)\cdot |3-a|].
	\end{equation}
	After solving the recurrence relation (\ref{re2}) with initial condition, we get (\ref{hxga}).
	
	When $a=0$, we have (\ref{hxgso}). When $a=1$, we have (\ref{hxgred}). 
	Since $|V(HXG(n; p_1, p_2))|=4n+2$, $|E(HXG(n; p_1,p_2))|=5n+1$, we have $2<\bar{d}=\frac{5n+1}{2n+1}<3$. For given $n$, $\bar{d}$ is a constant and therefore we get (\ref{hxgavr}).
\end{proof}

Let $\mathcal{P}_n=HXG(n;0,1)$ and $\mathcal{R}_n=HXG(n;p_1,0)$, where $p_1\in [0,1]$. By Theorem \ref{hxg}, we have

\begin{cor}
	The Sombor indices of $\mathcal{R}_n$ and $\mathcal{P}_n$ are
$$SO_{a}(\mathcal{R}_n)=2n\sqrt{2a^2-10a+13}+\sqrt{2}(n+4)\cdot |2-a|+\sqrt{2}(2n-3)\cdot |3-a|,$$
$$SO_{a}(\mathcal{P}_n)=4(n-1)\sqrt{2a^2-10a+13}+6\sqrt{2}\cdot |2-a|+\sqrt{2}(n-1)\cdot |3-a|.$$
\end{cor}

\begin{cor}
	Among all random hexagonal chains $HXG_{n} (n\geq 2)$, we have\\
(1) $(8\sqrt{2}+2\sqrt{13})n-\sqrt{2}  \leq SO(HXG_{n})\leq (3\sqrt{2}+4\sqrt{13})n+9\sqrt{2}-4\sqrt{13}$, with left equality iff $G\cong \mathcal{R}_n$, right equality iff $G\cong \mathcal{P}_n$.\\	
(2) $(5\sqrt{2}+2\sqrt{5})n-2\sqrt{2} \leq SO_{red}(HXG_{n})\leq (2\sqrt{2}+4\sqrt{5})n+4\sqrt{2}-4\sqrt{5},$ with left equality iff $G\cong \mathcal{R}_n$, right equality iff $G\cong \mathcal{P}_n$.\\	
(3) $2n\sqrt{2\bar{d}^2-10\bar{d}+13}+\sqrt{2}[(4-\bar{d})n+7\bar{d}-17] \leq SO_{avr}(HXG_{n}) \leq 4(n-1)\sqrt{2\bar{d}^2-10\bar{d}+13}+\sqrt{2}[(3-\bar{d})n+7\bar{d}-15],\  where \ \bar{d}=\frac{5n+1}{2n+1},$	with left equality iff $G\cong \mathcal{R}_n$, right equality iff $G\cong \mathcal{P}_n$.
\end{cor}
\begin{proof}
	Since $E_n=(n-2)(2\sqrt{13}-5\sqrt{2})p_2+\sqrt{2}(8n-1)+2\sqrt{13}n$ and $(n-2)(2\sqrt{13}-5\sqrt{2})\geq 0$, $SO(HXG_n)$ reaches the maximum value when $p_2=1$ and reaches the minimum value when $p_2=0$.
	
	Since $E_n^{red}=(n-2)(2\sqrt{5}-3\sqrt{2})p_2+\sqrt{2}(5n-2)+2\sqrt{5}n$ and $(n-2)(2\sqrt{5}-3\sqrt{2})\geq 0$, $SO_{red}(HXG_n)$ reaches the maximum value when $p_2=1$ and reaches the minimum value when $p_2=0$.
	
	$E_n^{avr}=(n-2)[2\sqrt{2\bar{d}^2-10\bar{d}+13}-\sqrt{2}]p_2+2n\sqrt{2\bar{d}^2-10\bar{d}+13}+\sqrt{2}[(4-\bar{d})n+7\bar{d}-17]$ can be regarded as a linear function of $p_2$. Since $n\geq 2$,  $2\bar{d}^2-10\bar{d}+13=2(\bar{d}-\frac{5}{2})^2+\frac{1}{2}\geq \frac{1}{2}$, we have $2\sqrt{2\bar{d}^2-10\bar{d}+13}-\sqrt{2}\geq 0$. Thus $SO_{avr}(HXG_n)$ reaches the maximum value when $p_2=1$ and reaches the minimum value when $p_2=0$.	
\end{proof}

Denote by $\mathcal{HC}_{n}$ the set of all hexagonal chains with $n$ hexagons. The average value of Sombor indices among $\mathcal{HC}_{n}$ can be characterized as
$$ A_{a}(\mathcal{HC}_{n})=\frac{1}{|\mathcal{HC}_{n}|}\sum_{G\in \mathcal{HC}_{n}}SO_a(G).$$
Since each element in $\mathcal{HC}_{n}$ has the same probability of occurrence, we have $p_{1}=p_{2}=1-p_{1}-p_{2}=\frac{1}{3}$. Then we have the following theorem.

\begin{theo}
	The average values of Sombor indices among $\mathcal{HC}_{n}$ are
	$$A_a(\mathcal{HC}_{n})=\frac{4}{3}(2n-1)\sqrt{2a^2-10a+13}+\frac{2}{3}\sqrt{2}(n+7)\cdot |2-a|+\frac{\sqrt{2}}{3}(5n-7)\cdot |3-a|.$$
\end{theo}


\subsection{Random phenylene chains}
\hspace{1.5em}
The phenylene chains are a class of conjugated hydrocarbons consists of hexagons and squares connected in turn, which has unique physicochemical properties due to their aromatic and antiaromatic rings. In \cite{raza2020c,raza2021}, Raza studied the expected values of some indices such as sum-connectivity, harmonic, symmetric division, arithmetic bond connectivity and geometric indices in random phenylene chains. In the following, we will study the Sombor indices of phenylene chains which are special molecular graphs. A phenylene chain $RPH_{n}$ with $n$ hexagons can be regarded as a phenylene chain $RPH_{n-1}$ with $n-1$ hexagons to which a new terminal hexagon has been adjoined by two edges. For $n\geq 3$, the terminal hexagon can be attached in three ways, which results in the local arrangements, we describe as $RPH_{n}^{1}$, $RPH_{n}^{2}$, and $RPH_{n}^{3}$, respectively (see Figure \ref{g1}).


\begin{figure}[h]
	\centering
	\includegraphics[scale=0.09]{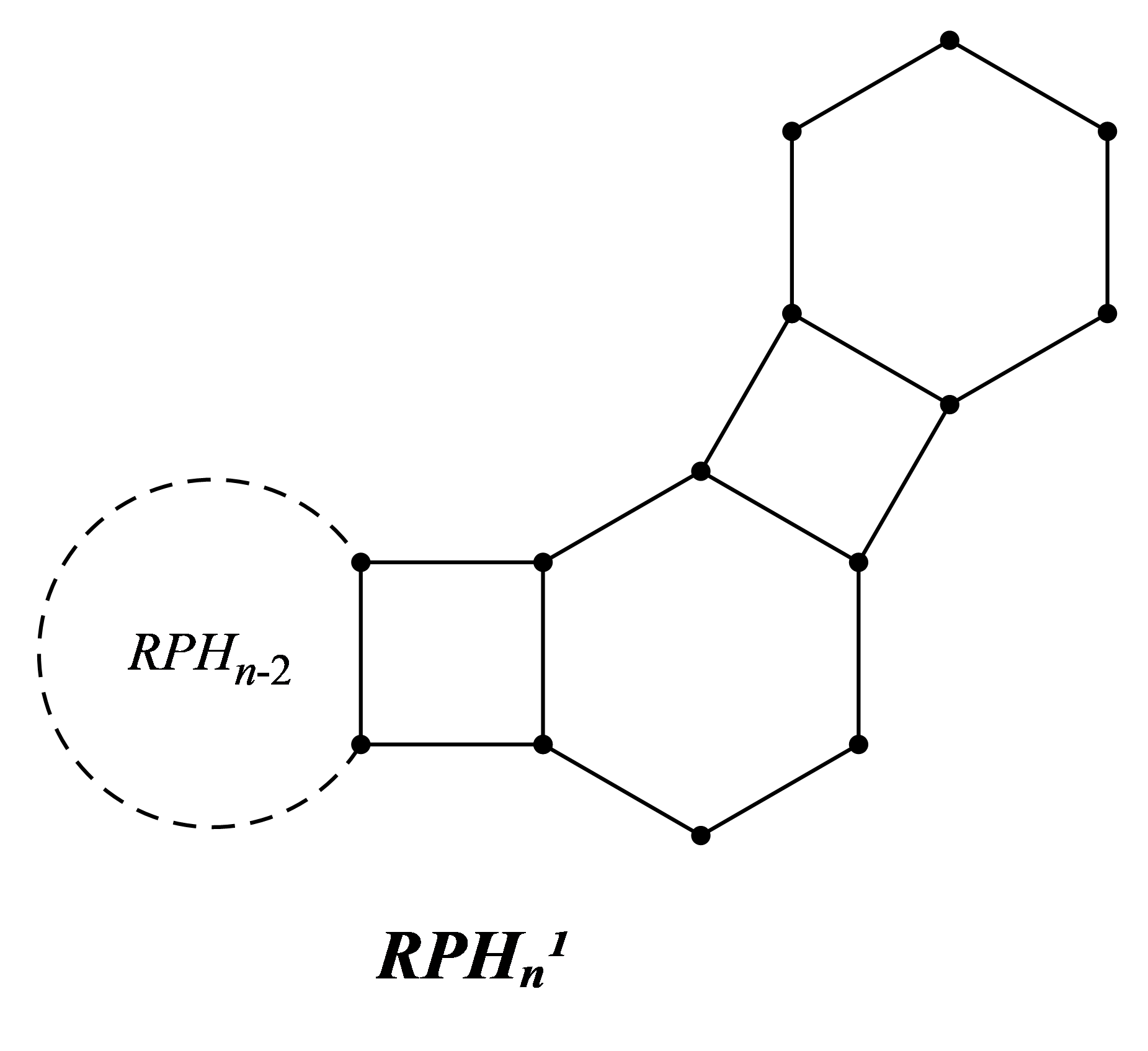}
	\qquad
	\includegraphics[scale=0.09]{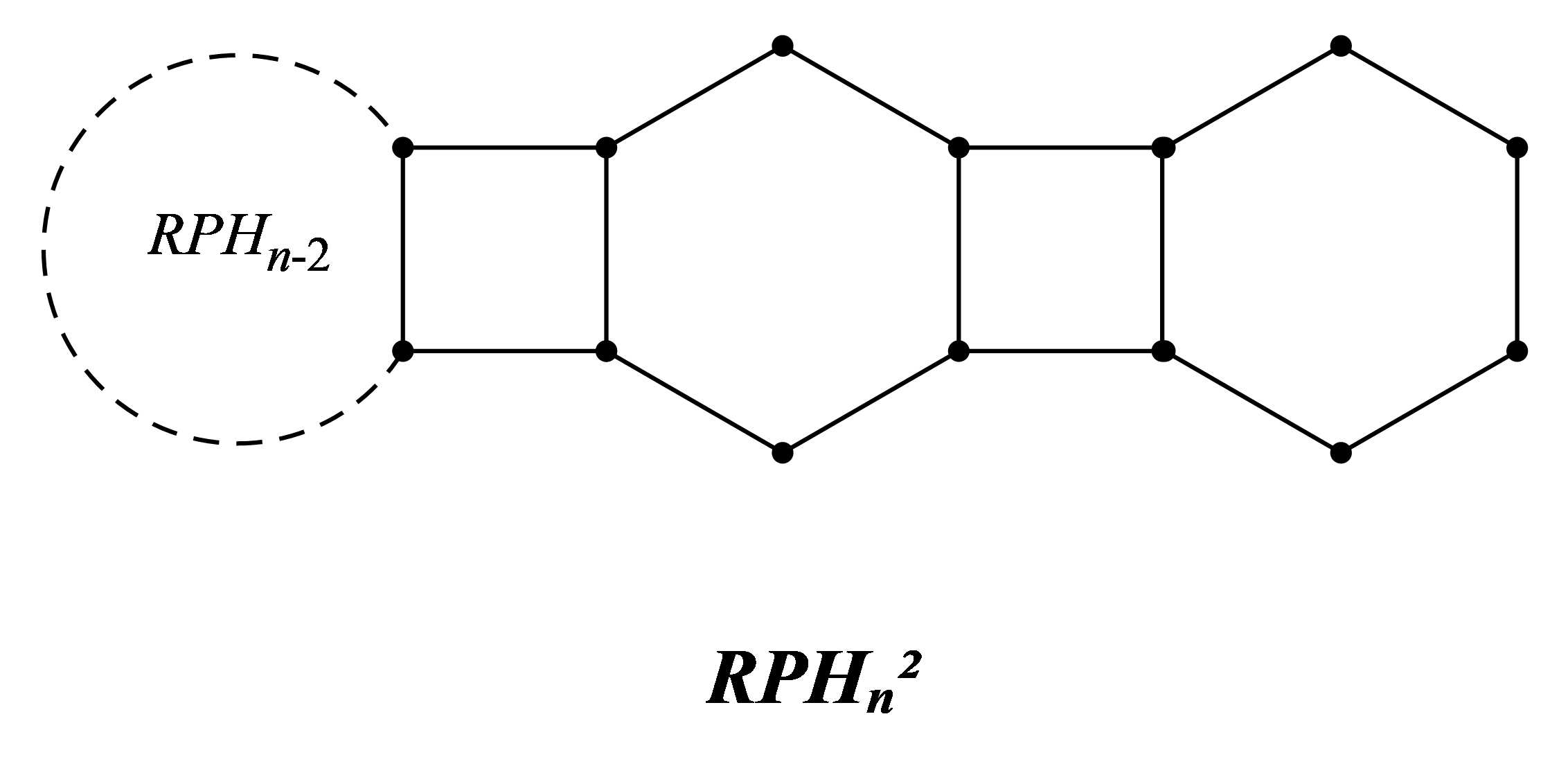}
	\qquad
	\includegraphics[scale=0.09]{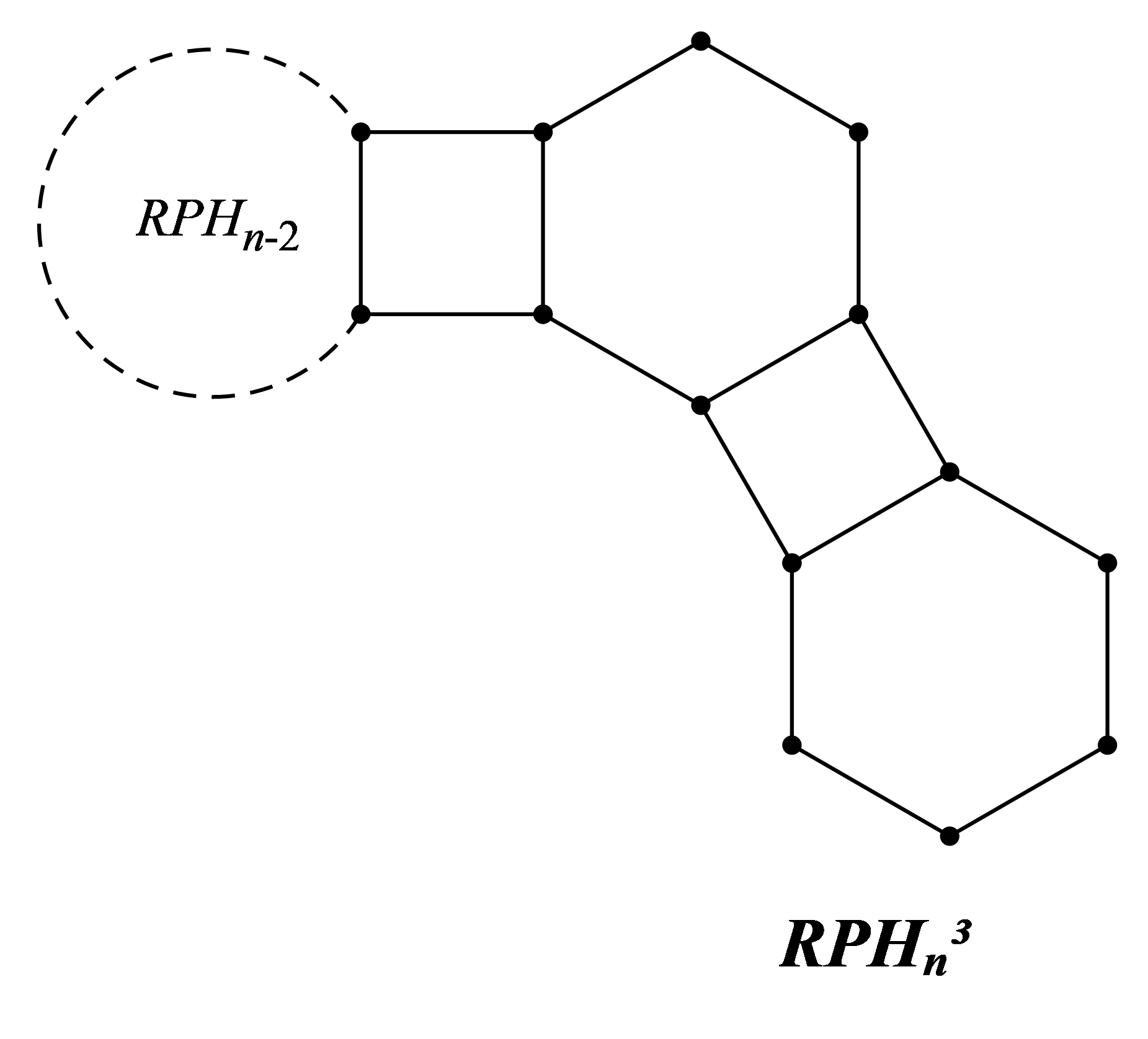}
	\caption{The three types of local arrangements in the random phenylene chains.}
	\label{g1}
\end{figure}

A random phenylene chain $RPH(n; p_{1},p_{2})$ with $n$ hexagons is a polyphenyl chain obtained by stepwise addition of terminal hexagons. At each step $t(=3,4,\cdots,n)$, a random selection is made from one of the three possible constructions: \\
$(1)$ $RPH_{t-1}\rightarrow RPH_{t}^{1}$ with probability $p_{1}$; \\
$(2)$ $RPH_{t-1}\rightarrow  RPH_{t}^{2}$ with probability $p_{2}$; \\
$(3)$ $RPH_{t-1}\rightarrow  RPH_{t}^{3}$ with probability $1-p_{1}-p_{2}$, where $p_{1}$, $p_{2}$ are constants, irrelative to the step parameter $t$.

We denote the expected values of Sombor indices by $\mathbb{E}_n^{a}=E[SO_a(RPH(n; p_1,p_2))]$, $\mathbb{E}_n=E[SO(RPH(n; p_1,p_2))]$, $\mathbb{E}_{n}^{red}=E[SO_{red}(RPH(n; p_1,p_2))]$ and $\mathbb{E}_{n}^{avr}=E[SO_{avr}(RPH(n; p_1,p_2))]$. In this section, $a$ is a constant.

\begin{theo}\label{rph}
	Let $RPH(n; p_1,p_2)$ be the random phenylene chain of length $n(\geq 2)$. Then
	\begin{equation}\label{rpha}
		\begin{split}
		\mathbb{E}_n^a=2(np_2-2p_2+n)&\sqrt{2a^2-10a+13}+\sqrt{2}(-np_2+2p_2+n+4)\cdot |2-a|\\
		&+\sqrt{2}(-np_2+2p_2+5n-6)\cdot |3-a|,
		\end{split}
	\end{equation}
	\begin{equation}\label{rphso}
		\mathbb{E}_n=[(2\sqrt{13}-5\sqrt{2})p_2+17\sqrt{2}+2\sqrt{13}]n+2(5\sqrt{2}-2\sqrt{13})p_2-10\sqrt{2},
	\end{equation}
	\begin{equation}\label{rphred}
		\mathbb{E}_{n}^{red}=[(2\sqrt{5}-3\sqrt{2})p_2+11\sqrt{2}+2\sqrt{5}]n+(6\sqrt{2}-4\sqrt{5})p_2-8\sqrt{2},
	\end{equation}
	\begin{equation}\label{rphavr}
			\begin{split}
		\mathbb{E}_n^{avr}=[2(p_2+1)\sqrt{2\bar{d}^2-10\bar{d}+13}+\sqrt{2}(13-p_2-4\bar{d})]n \\
		-4p_2 \sqrt{2\bar{d}^2-10\bar{d}+13}+2\sqrt{2}(p_2+5\bar{d}-13),\ where \ \bar{d}=\frac{8n-2}{3n}.
		\end{split}
	\end{equation}
\end{theo}
\begin{proof} 
	From the structure of the phenylene chain, it is easy to see that there exists only $(2,2)$, $(2,3)$ and $(3,3)$-type of edges. From Proposition \ref{main}, when $n=2$, $\mathbb{E}_2^{a}=4\sqrt{2a^2-10a+13}+6\sqrt{2}\cdot |2-a|+4\sqrt{2}\cdot |3-a|$. For $n\geq 3$, there are three possibilities to be considered (see Figure \ref{g1}).
	
	\textbf{Case 1}. $RPH_{n-1}\rightarrow RPH_n^1$.
	$$m_{22}(RPH_n^1)=m_{22}(RPH_{n-1})+1;$$
	$$m_{23}(RPH_n^1)=m_{23}(RPH_{n-1})+2;$$
	$$m_{33}(RPH_n^1)=m_{33}(RPH_{n-1})+5.$$
	Thus, $SO_{a}(RPH_n^1)=SO_{a}(RPH_{n-1})+2\sqrt{2a^2-10a+13}+\sqrt{2}\cdot |2-a|+5\sqrt{2}\cdot |3-a|.$
	
	\textbf{Case 2}. $RPH_{n-1}\rightarrow RPH_n^2$.
	$$m_{22}(RPH_n^2)=m_{22}(RPH_{n-1})+0;$$
	$$m_{23}(RPH_n^2)=m_{23}(RPH_{n-1})+4;$$
	$$m_{33}(RPH_n^2)=m_{33}(RPH_{n-1})+4.$$
	Thus, $SO_{a}(RPH_n^2)=SO_{a}(RPH_{n-1})+4\sqrt{2a^2-10a+13}+4\sqrt{2}\cdot |3-a|.$
	
	\textbf{Case 3}. $RPH_{n-1}\rightarrow RPH_n^3$.
	$$m_{22}(RPH_n^3)=m_{22}(RPH_{n-1})+1;$$
	$$m_{23}(RPH_n^3)=m_{23}(RPH_{n-1})+2;$$
	$$m_{33}(RPH_n^3)=m_{33}(RPH_{n-1})+5.$$
	Thus, $SO_{a}(RPH_n^3)=SO_{a}(RPH_{n-1})+2\sqrt{2a^2-10a+13}+\sqrt{2}\cdot |2-a|+5\sqrt{2}\cdot |3-a|.$
	
	Therefore, $\mathbb{E}_{n}^{a}=p_1 \cdot SO_a(RPH_n^1)+p_2\cdot  SO_a(RPH_n^2)+(1-p_1-p_2)\cdot SO_a(RPH_n^3)=SO_a(RPH_{n-1})+2(p_2+1)\sqrt{2a^2-10a+13}+\sqrt{2}(1-p_2)\cdot |2-a|+\sqrt{2}(5-p_2)\cdot |3-a|$. Since $E[\mathbb{E}_n^{a}]=\mathbb{E}_n^{a}$, we have
	\begin{equation}\label{rec2}
	\mathbb{E}_n^{a}=\mathbb{E}_{n-1}^{a}+2(p_2+1)\sqrt{2a^2-10a+13}+\sqrt{2}(1-p_2)\cdot |2-a|+\sqrt{2}(5-p_2)\cdot |3-a|.
	\end{equation}
	After solving the recurrence relation (\ref{rec2}) with initial condition, we get (\ref{rpha}).
	
When $a=0$, we have (\ref{rphso}). When $a=1$, we have (\ref{rphred}). Since $|V(RPH(n; p_1, p_2))|=6n$, $|E(RPH(n; p_1,p_2))|=8n-2$, we have $2<\bar{d}=\frac{8n-2}{3n}<3$. For given $n$, $\bar{d}$ is a constant and therefore we get (\ref{rphavr}).
\end{proof}

Let $\mathbb{R}_n=RPH(n;p_1,0)$, where $p_1\in [0,1]$ and $\mathbb{P}_n=RPH(n;0,1)$. By Theorem \ref{hxg}, we have

\begin{cor}
	The Sombor indices of $\mathbb{R}_n$ and $\mathbb{P}_n$ are	

	$SO_a(\mathbb{R}_n)=2n\sqrt{2a^2-10a+13}+\sqrt{2}(n+4)\cdot |2-a|+\sqrt{2}(5n-6)\cdot |3-a|,$
	
	 $SO_a(\mathbb{P}_n)=4(n-1)\sqrt{2a^2-10a+13}+6\sqrt{2}\cdot |2-a|+4\sqrt{2}(n-1)\cdot |3-a|.$
\end{cor}

\begin{cor}
	Among all random phenylene chains $RPH_{n} (n\geq 2)$, we have\\
(1) $ (17\sqrt{2}+2\sqrt{13})n-10\sqrt{2}  \leq SO(RPH_{n})\leq (12\sqrt{2}+4\sqrt{13})n-4\sqrt{13},$ with left equality iff $G\cong \mathbb{R}_n$, right equality iff $G\cong \mathbb{P}_n$.	\\
(2) $ (11\sqrt{2}+2\sqrt{5})n-8\sqrt{2} \leq SO_{red}(RPH_{n})\leq (8\sqrt{2}+4\sqrt{5})n-2\sqrt{2}-4\sqrt{5},$ with left equality iff $G\cong \mathbb{R}_n$, right equality iff $G\cong \mathbb{P}_n$.	\\
(3)	$2n\sqrt{2\bar{d}^2-10\bar{d}+13}+\sqrt{2}[(13-4\bar{d})n+2(5\bar{d}-13)] \leq SO_{avr}(RPH_{n}) \leq  4(n-1)\sqrt{2\bar{d}^2-10\bar{d}+13}+2\sqrt{2}[2n(3-\bar{d})+5\bar{d}-12], \ where \ \bar{d}=\frac{8n-2}{3n},$ with left equality iff $G\cong \mathbb{R}_n$, right equality iff $G\cong \mathbb{P}_n$.	
\end{cor}
\begin{proof}
	Since $\mathbb{E}_n=(n-2)(2\sqrt{13}-5\sqrt{2})p_2+\sqrt{2}(17n-10)+2\sqrt{13}n$ and $(n-2)(2\sqrt{13}-5\sqrt{2})\geq 0$, $SO(RPH_n)$ reaches the maximum value when $p_2=1$ and reaches the minimum value when $p_2=0$.
	
	Since $\mathbb{E}_{n}^{red}=(n-2)(2\sqrt{5}-3\sqrt{2})p_2+\sqrt{2}(11n-8)+2\sqrt{5}n$ and $(n-2)(2\sqrt{5}-3\sqrt{2})\geq 0$, $SO_{red}(RPH_n)$ reaches the maximum value when $p_2=1$ and reaches the minimum value when $p_2=0$.
	
	$\mathbb{E}_n^{avr}=(n-2)[2\sqrt{2\bar{d}^2-10\bar{d}+13}-\sqrt{2}]p_2+2n\sqrt{2\bar{d}^2-10\bar{d}+13}+\sqrt{2}[(13-4\bar{d})n+2(5\bar{d}-13)]$ can be regarded as a linear function of $p_2$. Since $n\geq 2$,  $2\bar{d}^2-10\bar{d}+13=2(\bar{d}-\frac{5}{2})^2+\frac{1}{2}\geq \frac{1}{2}$, we have $2\sqrt{2\bar{d}^2-10\bar{d}+13}-\sqrt{2}\geq 0$. Thus $SO_{avr}(RPH_n)$ reaches the maximum value when $p_2=1$ and reaches the minimum value when $p_2=0$.	
\end{proof}

Denote by $\mathcal{PC}_{n}$ the set of all phenylene chains with $n$ hexagons. The average value of Sombor indices among $\mathcal{PC}_{n}$ can be characterized as
$$ A_a(\mathcal{PC}_{n})=\frac{1}{|\mathcal{PC}_{n}|}\sum_{G\in \mathcal{PC}_{n}}SO_a(G).$$
Since each element in $\mathcal{PC}_{n}$ has the same probability of occurrence, we have $p_{1}=p_{2}=1-p_{1}-p_{2}=\frac{1}{3}$. Then we have the following theorem.

\begin{theo}
	The average values of Sombor indices among $\mathcal{PC}_{n}$ are
	$$A_a(\mathcal{PC}_{n})=\frac{4}{3}(2n-1)\sqrt{2a^2-10a+13}+\frac{2}{3}\sqrt{2}(n+7)\cdot |2-a|+\frac{2}{3}\sqrt{2}(7n-8)\cdot |3-a|.$$
\end{theo}

\subsection{Comparisons between Sombor indices with respect to random hexagonal chains and random phenylene chains}
\hspace{1.5em}
With the help of Theorems \ref{hxg} and \ref{rph}, we make a comparison between the expected values for Sombor index, reduced Sombor index and average Sombor index of a random hexagonal chain or a random phenylene chain with the same probabilities $p_i \ (i=1,2)$ (see Figure \ref{compar1}, \ref{compar2}).

\begin{theo}
 Let $HXG(n; p_1,p_2)$ be the hexagonal chain of length $n(\geq 2)$ and $RPH(n; p_1,p_2)$ be the random phenylene chain of length $n(\geq 2)$. Then
 $$E[SO(G)]>E[SO_{red}(G)]>E[SO_{avr}(G)], \mbox{ where } G\cong HXG(n; p_1,p_2) \mbox{ or } RPH(n; p_1,p_2),$$
 $$E[SO(RPH(n; p_1,p_2))]>E[SO(HXG(n; p_1,p_2))],$$
 $$E[SO_{red}(RPH(n; p_1,p_2))]>E[SO_{red}(HXG(n; p_1,p_2))],$$
 $$E[SO_{avr}(RPH(n; p_1,p_2))]>E[SO_{avr}(HXG(n; p_1,p_2))].$$
\end{theo}
\begin{proof}
	Since $2 \leq d_i, d_j\leq 3$, $2<\bar{d}<3$, we have $$\sqrt{d_i^2+d_j^2}>\sqrt{(d_i-1)^2+(d_j-1)^2}>\sqrt{(d_i-\bar{d})^2+(d_j-\bar{d})^2},$$ thus $E[SO(G)]>E[SO_{red}(G)]>E[SO_{avr}(G)]$.
	
	Since 
	$$E[SO(RPH(n; p_1,p_2))]-E[SO(HXG(n; p_1,p_2))]=9\sqrt{2}(n-1)>0,$$ $$E[SO_{red}(RPH(n; p_1,p_2))]-E[SO_{red}(HXG(n; p_1,p_2))]=6\sqrt{2}(n-1)>0,$$
	we have 
	$$E[SO(RPH(n; p_1,p_2))]>E[SO(HXG(n; p_1,p_2))],$$
	$$E[SO_{red}(RPH(n; p_1,p_2))]>E[SO_{red}(HXG(n; p_1,p_2))].$$ 
	
	When $n=2$, from Theorem \ref{hxg} and Theorem \ref{rph}, we have $E[SO_{avr}(RPH(2; p_1,p_2))]>E[SO_{avr}(HXG(2; p_1,p_2))]$. Let $\bar{d}_1=\bar{d}(HXG(n; p_1,p_2))$ and $\bar{d}_2=\bar{d}(RPH(n; p_1,p_2)).$ Since 
	$$\frac{11}{5}\leq \bar{d}_1=\frac{5n+1}{2n+1}<\frac{5}{2},\quad \frac{7}{3}\leq \bar{d}_2=\frac{8n-2}{3n}<\frac{8}{3},$$ 
	we have $\bar{d}_1-4\bar{d}_2\geq \frac{11}{5}-4 \cdot \frac{8}{3}=-\frac{127}{15}$.
	Let $f(a)=2\sqrt{2a^2-10a+13}$, then $f(\bar{d}_2)-f(\bar{d}_1)\geq f(\frac{5}{2})-f(\frac{11}{5})=\sqrt{2}-\frac{2}{5}\sqrt{17}$. By (\ref{re2}) and (\ref{rec2}), 
	\begin{equation*}
		\begin{split}
	&\mathbb{E}_n^{avr}-\mathbb{E}_{n-1}^{avr}-(E_n^{avr}-E_{n-1}^{avr})\\
	=&(p_2+1) \left( 2\sqrt{2\bar{d}_2^2-10\bar{d}_2+13}-2\sqrt{2\bar{d}_1^2-10\bar{d}_1+13} \right)+\sqrt{2}(9+\bar{d}_1-4\bar{d}_2)\\
	\geq & (p_2+1)\cdot (\sqrt{2}-\frac{2}{5}\sqrt{17})+\sqrt{2} \cdot (9-\frac{127}{15})\\
	\geq & 2(\sqrt{2}-\frac{2}{5}\sqrt{17})+\sqrt{2} \cdot (9-\frac{127}{15})\\
	=& (11-\frac{127}{15})\sqrt{2}-\frac{7}{5}\sqrt{17}>0.
	\end{split}
	\end{equation*}	
		
Therefore, $E[SO_{avr}(RPH(n; p_1,p_2))]>E[SO_{avr}(HXG(n; p_1,p_2))].$
\end{proof}

\begin{figure}[h]
	\centering
			\begin{minipage}{220pt}
				\centering
				\includegraphics[scale=0.5]{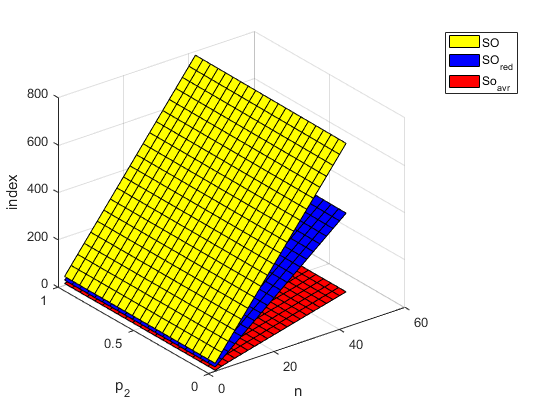}
				\caption{Difference between $SO(HXG_{n})$, $SO_{red}(HXG_{n})$ and $SO_{avr}(HXG_{n})$}
				\label{compar1}
		\end{minipage}
		\qquad
		\begin{minipage}{225pt}
			\centering
			\includegraphics[scale=0.5]{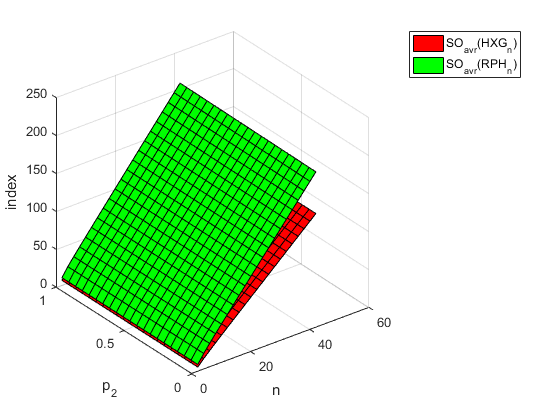}
			\caption{Difference between $SO_{avr}(HXG_{n})$ and $SO_{avr}(RPH_{n})$}
			\label{compar2}
		\end{minipage}
\end{figure}

\section{The Sombor indices of graphene, coronoid systems and carbon nanocones}\label{sec3}
\hspace{1.5em}
Topological indices are important graph invariants used for describing various properties of molecules. Methods for computing topological indices of some molecular graphs such as benzenoid systems, phenylenes or coronoid systems were studied in \cite{bret2021,tra2019,zig2019}. In this section, we study the Sombor indices of graphene, coronoid systems and carbon nanocones.

Graphene \cite{boes1986,novg2004}, denoted by $GN(n,k)$, is a flat monolayer of carbon atoms tightly packed into a two-dimensional hexagonal lattice that forms a basic building block for graphitic materials of different forms (see Figure \ref{n1}). Due to the C-C covalent bonds, graphene is the hardest material known in nature \cite{gein2007}. There are various results about the topological indices of graphene in recent years \cite{aroc2016,aroc2020}.

\begin{figure}[h]
	\centering
	\scalebox{.08}[.08]{\includegraphics{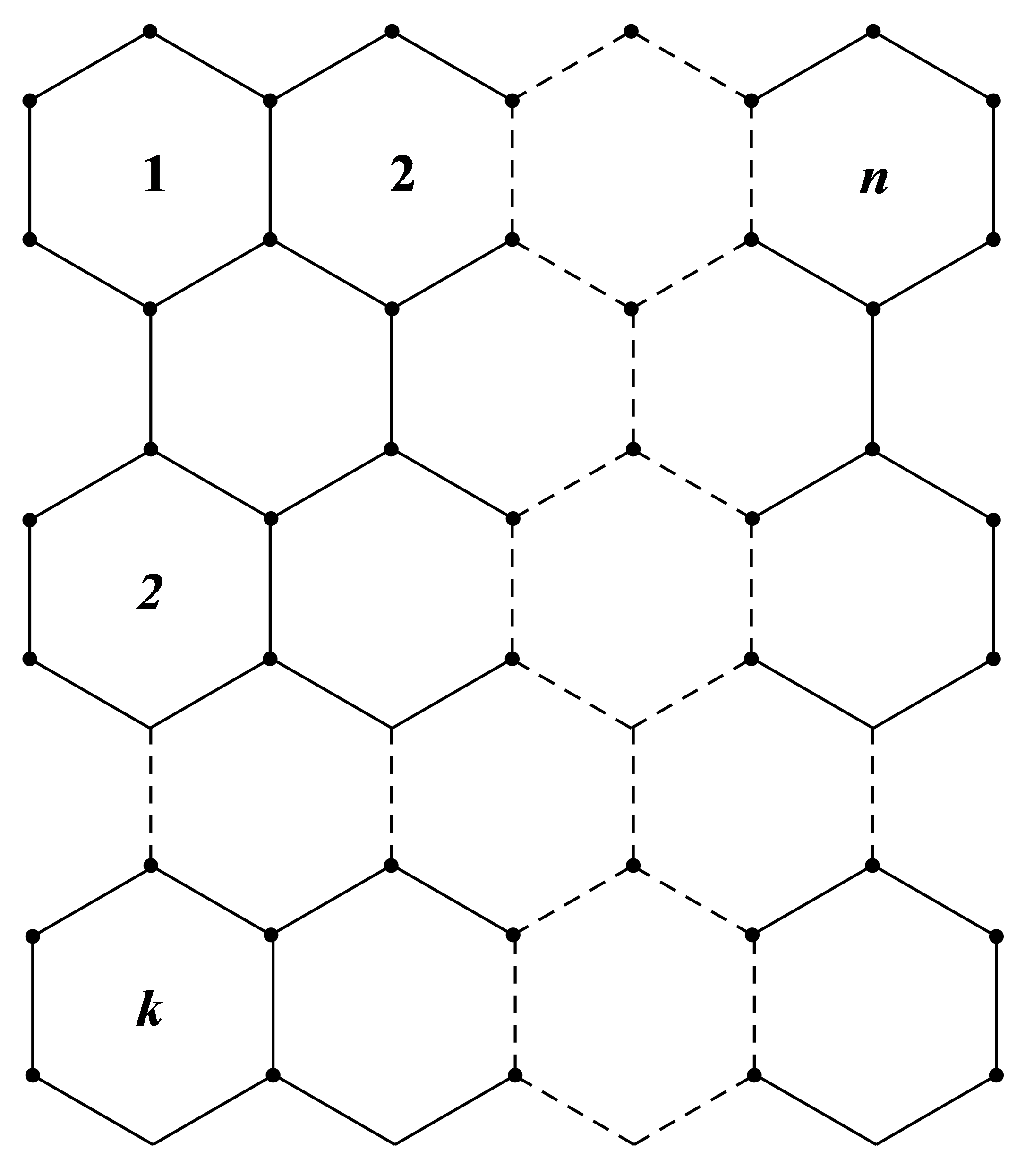}}
	\caption{Structure of graphene $GN(n,k)$.}
	\label{n1}
\end{figure}
\begin{theo}
	Let $G$ be the graphene $GN(n,k)$, $1\leq k\leq n$. Then
	$$SO(G)=4\sqrt{13}(n+k-2)+\sqrt{2}(18nk-15n-11k+20),$$
	$$SO_{red}(G)=4\sqrt{5}(n+k-2)+\sqrt{2}(12nk-8k-10n+12),$$
	\begin{align*}
	SO_{avr}&(G)=\frac{4(k+n-2)}{k(2n+1)}\sqrt{4k^2n^2-4k^2n+5k^2-4kn^2+6kn+2n^2}\\
	&+\frac{\sqrt{2}(16k^2n-12k^2+6kn^2-9kn+4k-5n^2)}{k(2n+1)}.
	\end{align*}
\end{theo}
\begin{proof}
	From the structure of graphene $GN(n,k)$, it is easy to see that there exists only $(2,2)$, $(2,3)$ and $(3,3)$-type of edges. Since $m_{22}(G)=2k+4$, $m_{23}(G)=4n+4k-8$, $m_{33}(G)=6nk-5k-5n+4$, from Proposition \ref{main}, we have
	$$SO_{a}(G)=4(n+k-2)\sqrt{2a^2-10a+13}+2\sqrt{2}(k+2)\cdot |2-a|+\sqrt{2}(6nk-5k-5n+4)\cdot |3-a|.$$
	Thus
	$$SO(G)=4\sqrt{13}(n+k-2)+\sqrt{2}(18nk-15n-11k+20),$$
	$$SO_{red}(G)=4\sqrt{5}(n+k-2)+\sqrt{2}(12nk-8k-10n+12).$$
	Since $|V(G)|=2(2n+1)k$, $|E(G)|=(6n+1)k-n$, we have $2<\bar{d}=\frac{(6n+1)k-n}{(2n+1)k}<3$. Thus
	\begin{align*}
			SO_{avr}&(G)=\frac{4(k+n-2)}{k(2n+1)}\sqrt{4k^2n^2-4k^2n+5k^2-4kn^2+6kn+2n^2}\\
			&+\frac{\sqrt{2}(16k^2n-12k^2+6kn^2-9kn+4k-5n^2)}{k(2n+1)}.
	\end{align*}

	The proof is completed.
\end{proof}

A coronoid system can be regarded as a benzenoid system that is allowed to have ‘holes’ such that the perimeter of the coronoid system and the perimeters of the holes are pairwise disjoint. There are many results on topological index of coronoid systems \cite{julv2020,crug2020}. We now consider a special family of coronoid systems, denoted by $K(n,p,r)$ (see Figure \ref{n2}), which is formally generated from polycyclic benzenoid systems by circumcising some interior atoms or bonds.

\begin{theo}\label{Knpr}
	Let $G$ be the $K(n,p,r)$ coronoid structure                                                                                                                                                                                                                                                                                                                                                                                                                                                                                                                                                                                                                                                                                                                                                                                                                                                                                                                                                                                                                                                                                                                                                                                                                                                                                                                                                                                                                                                                                                                                                                                                                                                                                                                                                                                                                                                                                                                                                                                                                                                                                                                                                                                                                                                                                                                                                                                                                                                                                                                                                                                                                                                                                                                                                                                                                                                                                                                                                                                                                                                                                                                                                                                                                                                                                                                                                                                                                                                                                                                                                                                                                                                                                                                                                                                                                                                                                   with $r\geq 1$, $n\geq 3$ and $1\leq p \leq n$. Then
	$$SO(G)=4\sqrt{13}(2n+4p+3r-6)+3\sqrt{2}[3(3r-2)(2p+n)+9r^2-15r+16],$$                             
	$$SO_{red}(G)= 4\sqrt{5}(2n+4p+3r-6)+\sqrt{2}[4(3r-2)(2p+n)+18r^2-30r+30],$$
	$$SO_{avr}(G)= \frac{1}{r+1}[4\sqrt{r^2+1}(2n+4p+3r-6)+\sqrt{2}(1+6r)].$$
\end{theo}

\begin{proof}
	From the structure of $K(n,p,r)$ coronoid structure, it is easy to see that there exists only $(2,2)$, $(2,3)$ and $(3,3)$-type of edges. Since $m_{22}(G)=6$, $m_{23}(G)=8(2p+n)+12(r-2)$, $m_{33}(G)=2(3r-2)(2p+n)+3(3r^2-5r+4)$, from Proposition \ref{main}, we have
	$$SO_{a}(G)=4\sqrt{2a^2-10a+13}(4p+2n+3r-6)+6\sqrt{2}\cdot |2-a|+|3-a|\cdot \sqrt{2}[2(3r-2)(2p+n)+3(3r^2-5r+4)].$$
	Thus
	$$SO(G)=4\sqrt{13}(2n+4p+3r-6)+3\sqrt{2}[3(3r-2)(2p+n)+9r^2-15r+16],$$
	$$SO_{red}(G)=4\sqrt{5}(4p+2n+3r-6)+\sqrt{2}[4(3r-2)(n+2p)+18r^2-30r+30].$$
	Since $|V(G)|=2(r+1)(4p+2n+3r-3)$, $|E(G)|=(3r+2)(4p+2n+3r-3)$, we have $2<\bar{d}=\frac{3r+2}{r+1}<3$. Thus
	$$SO_{avr}(G)=\frac{1}{r+1}[4\sqrt{r^2+1}(2n+4p+3r-6)+\sqrt{2}(1+6r)].$$

	The proof is completed.
\end{proof}

\begin{figure}[h]
	\centering
	\scalebox{.10}[.1]{\includegraphics{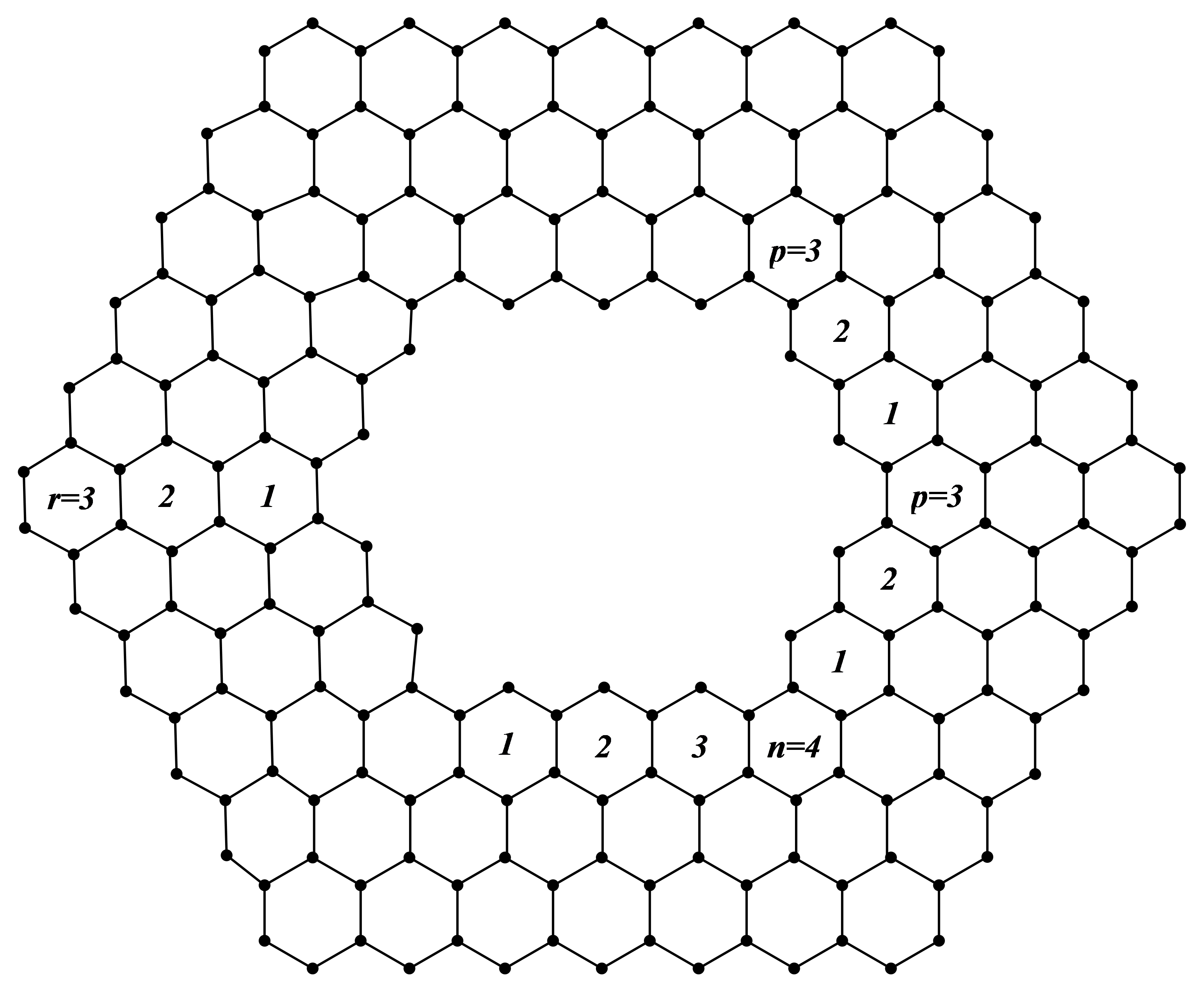}}
	\caption{Structure of coronoid system $K(n,p,r)$.}
	\label{n2}
\end{figure}

From Theorem \ref{Knpr}, it is easy to obtain Corollary \ref{K1} and Corollary \ref{K2} as special cases. More precisely, we use the previous theorem on $K(2,1,r)$ and $K(2,2,r)$ to compute the indices for $r$-circumscribed $C_{32}H_{16}$ and $C_{48}H_{24}$ coronoid structures.

\begin{cor}\label{K1}
Let $G$ be an $r$-circumscribed $C_{32}H_{16}$ coronoid structure {\rm ($r\geq 1$)}. Then
$$SO(G)= 4\sqrt{13}(2+3r)+3\sqrt{2}(9r^2+21r-8) ,$$                             
$$SO_{red}(G)= 4\sqrt{5}(2+3r)+2\sqrt{2}(9r^2+9r-1),$$
$$SO_{avr}(G)= \frac{1}{r+1}[4\sqrt{r^2+1}(2+3r)+\sqrt{2}(1+6r)].$$
\end{cor}

\begin{cor}\label{K2}
Let $G$ be an $r$-circumscribed $C_{48}H_{24}$ coronoid structure {\rm ($r\geq 1$)}. Then
$$SO(G)= 12\sqrt{13}(2+r)+3\sqrt{2}(9r^2+39r-20) ,$$                             
$$SO_{red}(G)= 12\sqrt{5}(2+r)+6\sqrt{2}(3r^2+7r-3),$$
$$SO_{avr}(G)= \frac{1}{r+1}[12\sqrt{r^2+1}(2+r)+\sqrt{2}(1+6r)].$$
\end{cor}

Carbon nanocones, denoted by $CNC_k(n)$, are conical structures, which are conceived as curved forms of graphite sheet obtained by excising a wedge and subsequently joining the edges (see Figure \ref{n3}). Carbon nanocones have a wide range of applications, such as caping ultrafine gold needles, which attracted the attention of both theoretical and experimental chemists. There are many results on topological index of carbon nanocones \cite{aroc2019,nazf2018}.

\begin{figure}[h]
	\centering
	\scalebox{.1}[.1]{\includegraphics{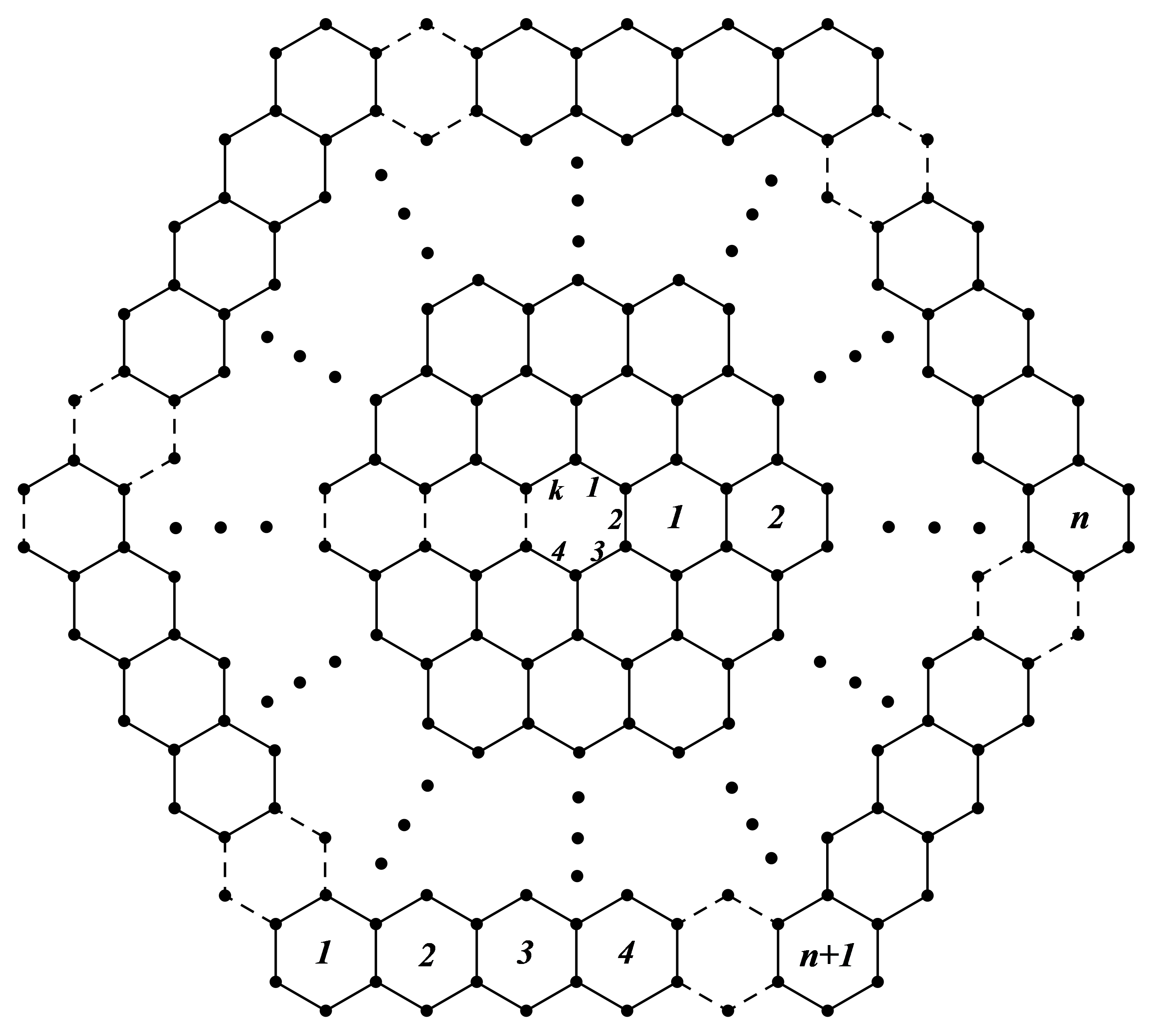}}
	\caption{Structure of carbon nanocone $CNC_k(n)$.}
	\label{n3}
\end{figure}

\begin{theo}
Let $G$ be the carbon nanocone structure $CNC_k(n)$ with $k>4$ and $n \geq 1$. Then
	$$SO(G)=2\sqrt{13}kn+\frac{\sqrt{2}k}{2}(9n^2+3n+4),$$      
	$$SO_{red}(G)= 2\sqrt{5}kn+\sqrt{2}k(3n^2+n+1),$$
	$$SO_{avr}(G)=\frac{kn}{2}\left(\frac{4}{n+1}\sqrt{n^2+1}+3\sqrt{2}\right).$$
\end{theo}
\begin{proof}
	From the structure of carbon nanocone $CNC_k(n)$, it is easy to see that there exists only $(2,2)$, $(2,3)$ and $(3,3)$-type of edges. Since $m_{22}(G)=k$, $m_{23}(G)=2kn$, $m_{33}(G)=kn(3n+1)/2$, from Proposition \ref{main}, we have
	$$SO_{a}(G)=2kn\sqrt{2a^2-10a+13}+\frac{\sqrt{2}k}{2}[2\cdot |2-a|+|3-a|\cdot (3n+1)n].$$
	Thus,
	$$SO(G)=2\sqrt{13}kn+\frac{\sqrt{2}k}{2}(9n^2+3n+4),$$
	$$SO_{red}(G)=2\sqrt{5}kn+\sqrt{2}k(3n^2+n+1).$$
	Since $|V(G)|=k(n+1)^2$, $|E(G)|=k(n+1)(3n+2)/2$, we have $2<\bar{d}=\frac{3n+2}{n+1}<3$. Thus
	$$SO_{avr}(G)=\frac{kn}{2}\left(\frac{4}{n+1}\sqrt{n^2+1}+3\sqrt{2}\right).$$

	The proof is completed.
\end{proof}
\begin{cor}
	Let $G$ be $n$-circumscribed one pentagonal carbon nanocone structure $CNC_5(n)$ with $n\geq 1$. Then
	$$SO(G)=10\sqrt{13}n+\frac{5\sqrt{2}}{2}(9n^2+3n+4),$$      
	$$SO_{red}(G)= 10\sqrt{5}n+5\sqrt{2}(3n^2+n+1),$$
	$$SO_{avr}(G)=\frac{5n}{2}\left(\frac{4}{n+1}\sqrt{n^2+1}+3\sqrt{2}\right).$$
\end{cor}
\begin{figure}[h!]
	\centering
	\begin{minipage}{220pt}
		\centering
		\includegraphics[scale=0.5]{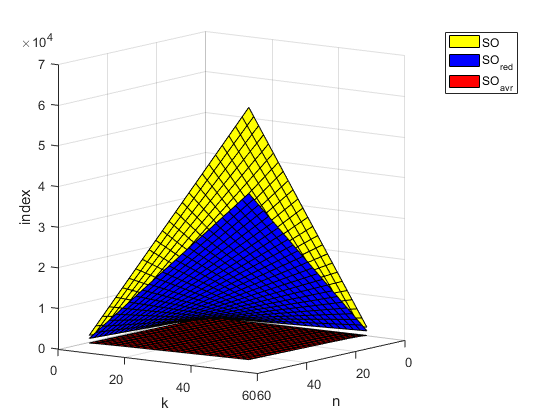}
		\caption{Differences between Sombor, reduced Sombor and average Sombor indices of $GN(n,k)$.}
		\label{gn}
	\end{minipage}
	\qquad
	\begin{minipage}{220pt}
		\centering
		\includegraphics[scale=0.5]{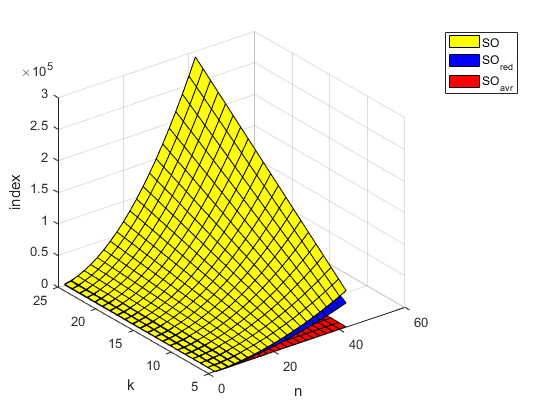}
		\caption{Differences between Sombor, reduced Sombor and average Sombor indices of $CNC_k(n)$.}
		\label{cnc}
	\end{minipage}
\end{figure}
\begin{table}[h!]
	\centering
	\begin{tabular}{llll}
		\hline
		$(n,p,r)$    & $SO(G)$  & $SO_{red}(G)$ & $SO_{avr}(G)$ \\
		\hline
		(3,1,1)&	207.02&	  116.35&	24.75\\
		(3,1,2)&	492.12&   261.98&	35.94\\
		(3,1,3)&	853.58&   458.51&	47.83\\
		(4,2,4)&	1759.79&  929.34&	79.64\\
		(4,2,5)&	2388.54&  1278.61&  92.29\\
		(4,2,6)&	3093.66&  1678.80&	104.80\\
		(5,2,1)&	373.31&   210.53& 	47.38\\
		(5,3,2)&	970.66&	  505.07&	71.72\\
		(5,4,3)&	1797.10&  918.41&	98.42\\
		(6,4,4)&	2696.53&  1376.08& 	119.22\\
		(6,4,5)&	3554.38&  1827.18&	133.08\\
		(6,4,6)&	4488.60&  2329.18&	146.51\\
		(9,5,7)&	6852.57&  3508.95&	194.98\\
		(9,6,8)&	8748.16&  4482.31&  222.69\\
		(9,7,9)&	10872.85& 5574.47&	250.46\\
		\hline
	\end{tabular}
	\caption{Numeric differences for $K(n,p,r)$}
	\label{knpr}
\end{table}

It's clear that $E[SO(G)]>E[SO_{red}(G)]>E[SO_{avr}(G)]$ for graph $G$ with only $(2,2)$, $(2,3)$ and $(3,3)$-type of edges. The graphical profiles of the comparison between Sombor, reduced Sombor and average Sombor indices of graphene $GN(n,k)$ or carbon nanocone structure $CNC_k(n)$ is give in Figure \ref{gn}, \ref{cnc}. The numerical comparison of the Sombor indices with respect to different types of $K(n,p,r)$ coronoid structure is give in Table \ref{knpr}.

\section{Conclusion}
\hspace{1.5em}
In this paper, the expected values of Sombor index, reduced Sombor index and average Sombor index have been determined for random hexagonal chains and random phenylene chains.  Explicit formulae for Sombor index, reduced Sombor index and average Sombor index of some chemical graphs such as graphene, coronoid systems and carbon nanocones are given. And detailed comparisons between these indices with respect to different chemical graphs have been determined explicitly. The structural characteristics of the compound can be deduced from the topological index formulae, which provides a theoretical basis for drug discovery and synthetic organic chemistry.

\end{document}